\documentclass{article}%
\usepackage{amsmath}
\usepackage{amsfonts}
\usepackage{amssymb}
\usepackage{graphicx}%
\providecommand{\U}[1]{\protect\rule{.1in}{.1in}}
\newtheorem{theorem}{Theorem}
\newtheorem{corollary}[theorem]{Corollary}
\newtheorem{lemma}[theorem]{Lemma}
\newtheorem{proposition}[theorem]{Proposition}
\newenvironment{proof}[1][Proof]{\noindent\textbf{#1.} }{\ \rule{0.5em}{0.5em}}
\begin{document}

\title{On groups of finite upper rank}
\author{Dan Segal}
\maketitle

\textbf{Rank and upper rank}

\medskip

For a finite group $G$ with Sylow $p$-subgroup $P$ the \emph{rank} and the
$p$\emph{-rank of }$G$ are defined by%
\begin{align*}
\mathrm{r}(G)  &  =\sup\{\mathrm{d}(H)\mid H\leq G\},\\
\mathrm{r}_{p}(G)  &  =\mathrm{r}(P),
\end{align*}
where as usual $\mathrm{d}(H)$ denotes the minimal size of a generating set
for $H$. When $G$ is an arbitrary group, $\mathcal{F}(G)$ denotes the set of
finite quotient groups of $G$, and we define the (`local' and `global')
\emph{upper ranks} of $G$:%
\begin{align*}
\mathrm{ur}_{p}(G)  &  =\sup\{\mathrm{r}_{p}(Q)\mid Q\in\mathcal{F}(G)\}\\
\mathrm{ur}(G)  &  =\sup\{\mathrm{r}(Q)\mid Q\in\mathcal{F}(G)\}.
\end{align*}
A theorem of Lucchini \cite{L}, first proved for soluble groups by Kov\'{a}cs
\cite{K}, asserts that for a finite group $G,$%
\[
\sup_{p}\mathrm{r}_{p}(G)\leq\mathrm{r}(G)\leq1+\sup_{p}\mathrm{r}_{p}(G),
\]
so the analogue holds for the upper ranks of an infinite group; in particular,
$\mathrm{ur}(G)$ is finite if and only if the local upper ranks $\mathrm{ur}%
_{p}(G)$ are bounded as $p$ ranges over all primes.

Let us denote by $\mathcal{U}$ \emph{the class of all groups} $G$ \emph{such
that} $\mathrm{ur}_{p}(G)$ \emph{is finite for every prime} $p$. One can
describe $\mathcal{U}$ more colourfully as \emph{the class of groups}
\emph{whose profinite completion has a} $p$\emph{-adic analytic Sylow pro-}$p$
\emph{subgroup for every prime} $p$ \cite{DDMS}.

\bigskip

\textbf{Background}

\medskip

More than 20 years ago, Alex Lubotzky conjectured that there is a `subgroup
growth gap' for finitely generated soluble groups. We had recently established
that a finitely generated (f.g.) residually finite group has polynomial
subgroup growth if and only if it is virtually a soluble minimax group (see
\cite{LMS} or \cite{LS}, Chapter 5). I showed in \cite{S} that there exist
f.g. groups of arbitrarily slow non-polynomial subgroup growth; the Lubotzky
question amounts to: do there exist such groups that are \emph{soluble}?

Now if a f.g. soluble group $G$ has subgroup growth of type strictly less than
$n^{\log n/(\log\log n)^{2}}$ then $\mathrm{ur}_{p}(G)$ is finite for every
prime $p$ (\cite{MS}, Prop. 2.6, \cite{S1}, Proposition C). On the other hand,
it is known that a finitely generated residually finite group has finite upper
rank if and only if it is virtually a soluble minimax group \cite{MS1}. So
Lubotzky's conjecture would follow from

\medskip

\textbf{Conjecture A} \cite{S1} \ \emph{Let }$G$ \emph{be a f.g. soluble
group. If }$G\in\mathcal{U}$ \emph{then} $G$ \emph{has finite upper rank.}

\bigskip

Equivalently: if the upper $p$-ranks of $G$ are all \emph{finite}, then they
are \emph{bounded}. If $G$ is assumed to be residually finite, this conclusion
is equivalent to saying that $G$ is a minimax group.

In fact, Conjecture A would imply that a f.g. soluble group canot have
subgroup growth of type strictly between polynomial and $n^{\log n}$
(\cite{S2}, Proposition 5.1).

I am now doubtful about this conjecture, having spent over two decades failing
to prove it. What follows is a survey of what is known on the topic.

\bigskip

\textbf{Olshanski-Osin groups}

\medskip

In \cite{MS} we raised the question: is Conjecture A true even without the
solubility hypothesis? If $G$ is a group with $\mathrm{ur}_{2}(G)$ finite then
$G$ has a subgroup $H$ of finite index such that every finite quotient of $H$
is soluble (\cite{LS}, Theorem 5.5.1). This (at first sight surprising)
consequence of the Odd Order Theorem suggests that the solubility hypothesis
in Conjecture A may be redundant. Without that hypothesis, however, the
conjecture is \emph{false}, as was recently pointed out to me by Denis Osin. I
am very grateful to him for allowing me to reproduce his argument here. It
depends on

\begin{theorem}
\emph{(\cite{OO} Theorem 1.2) }Let $P=(p_{i})$ be an infinite sequence of
primes. There exists an infinite $2$-generator periodic group $G(P)=G_{0}$
having a descending chain of normal subgroups $(G_{i})_{i\geq0}$ with $\bigcap
G_{i}=1$ such that $G_{i-1}/G_{i}$ is abelian of exponent dividing $p_{i}$ for
each $i\geq1$.
\end{theorem}

Now let $G=G(P)$ where $P$ consists of distinct primes. Each quotient
$G/G_{n}$ is finite. Given $m\in\mathbb{N}$ there exists $n$ such that
$p_{i}\nmid m$ for all $i\geq n$. It is easy to see that each element of
$G_{n}$ has order coprime to $m,$ whence $G_{n}\leq G^{m}$. It follows that
for each prime $p$,%
\begin{align*}
\mathrm{ur}_{p}(G)  &  =\sup\{\mathrm{ur}_{p}(G/G^{m})\mid m\in\mathbb{N}\}\\
&  =\sup\{\mathrm{r}_{p}(G/G_{n})\mid n\in\mathbb{N}\}=\left\{
\begin{array}
[c]{c}%
\mathrm{r}_{p}(G/G_{k})\text{ if }p=p_{k}\\
0\text{ if }p\neq p_{i}~\forall i
\end{array}
\right\}  <\infty.
\end{align*}

Thus $G\in\mathcal{U}$. On the other hand, $G$ is residually finite and not
virtually soluble (as it is infinite, f.g. and periodic), and so $G$ has
infinite upper rank by the theorem from \cite{MS1} quoted above.\medskip

Whether Conjecture A holds with `soluble' replaced by `torsion-free' is still
an open problem.

The groups of slow subgroup growth constructed in \cite{S} are built out of
finite simple groups. The groups $G(P),$ in contrast, have all their finite
quotients soluble: I call such groups of \emph{prosoluble type} (because their
profinite completions are prosoluble). As far as I know, these provide the
first such examples with arbitrarily slow non-polynomial subgroup growth; they
show that Lubotzky's conjecture becomes false if `soluble' is replaced by `of
prosoluble type':

\begin{proposition}
let $f:\mathbb{N}\rightarrow\mathbb{R}_{>0}$ be an unbounded non-decreasing
function. Then there exists a sequence $P$ of primes such that the group
$G=G(P)$ satisfies%
\[
s_{n}(G)\leq n^{f(n)}%
\]
for all large $n$, but $G$ does not have polynomial subgroup growth.
\end{proposition}

Here, $s_{n}(G)$ denotes the number of subgroups of index at most $n$ in $G$.

\bigskip

\begin{proof}
Suppose that $P=(p_{i})$ is a strictly increasing sequence of primes. Let $H$
be a proper subgroup of index $\leq n$ in $G$. Then $G_{0}>H\geq G^{n!}\geq
G_{k}$ for some $k.$ Let $k$ be minimal such. Then $G_{k}\leq H\cap
G_{k-1}<G_{k-1}$, so%
\[
p_{k}\mid\left\vert G_{k-1}:H\cap G_{k-1}\right\vert \leq n\text{.}%
\]
It follows that%
\[
s_{n}(G)=s_{n}(G/G_{k(n)})
\]
where $k(n)$ is the largest $k$ such that $p_{k}\leq n$.

Put $Q_{n}=G/G_{k(n)}$. According to \cite{LS}, Corollary 1.7.2,%
\[
s_{n}(Q_{n})\leq n^{2+r(n)}%
\]
where $r(n)=\max_{p}\mathrm{r}_{p}(Q_{n})$. Write $m_{j}=\left\vert
G:G_{j-1}\right\vert $ for each $j\geq1$. Since $G$ is a $2$-generator group,
$G_{j-1}$ can be generated by $1+m_{j}$ elements, and so $\mathrm{r}_{p_{j}%
}(Q_{n})\leq1+m_{j}$ for $j\leq k(n),$ while $\mathrm{r}_{p}(Q_{n})=0$ if
$p\notin\{p_{1},\ldots,p_{k(n)}\}$.

Now we can choose the sequence $P$ recursively as follows: $p_{1}$ is
arbitrary. Set $\mu_{1}=1$. Given $p_{i}$ and $\mu_{i}$ for $i\leq t$, set%
\[
\mu_{t+1}=\mu_{t}\cdot p_{t}^{1+\mu_{t}}%
\]
and let $p_{t+1}>p_{t}$ be a prime so large that%
\[
f(p_{t+1})\geq3+\mu_{t+1}.
\]

Note that $\left\vert G_{j-1}:G_{j}\right\vert \leq p_{j}^{1+m_{j}}$ for each
$j$, so $m_{j+1}\leq m_{j}\cdot p_{j}^{1+m_{j}}$. It follows that $m_{j}%
\leq\mu_{j}$ for all $j$. Then%
\begin{align*}
r(n)\leq\max\{1+m_{j}\mid j\leq k(n)\}  &  \leq\max\{1+\mu_{j}\mid j\leq
k(n)\}\\
&  =1+\mu_{k(n)}\\
&  \leq f(p_{k(n)})-2\leq f(n)-2.
\end{align*}
Thus
\[
s_{n}(G)=s_{n}(Q_{n})\leq n^{2+r(n)}\leq n^{f(n)}.
\]

Of course, $G$ does not have polynomial subgroup growth because it has
infinite upper rank, as observed above.
\end{proof}

\bigskip

\textbf{Minimax groups: a reminder}

\medskip

Let us denote by $\mathcal{S}$ the class of all residually finite virtually
soluble minimax groups. The following known results will be used without
special mention:

\begin{itemize}
\item If $G\in\mathcal{S}$ then $G$ is virtually nilpotent-by-abelian.

\item If $G\in\mathcal{S}$ then $G$ is virtually residually (finite nilpotent).

\item A minimax group is in $\mathcal{S}$ if and only if it is virtually torsion-free.

\item The class $\mathcal{S}$ is extension-closed.

\item If $G$ is f.g. and virtually residually nilpotent then $G$ is residually finite.

\item If $G$ has a nilpotent normal subgroup $N$ such that $G/N^{\prime}$ is
(a) minimax \emph{resp}. (b) of finite upper rank, then $G$ is (a) minimax,
\emph{resp}. (b) of finite upper rank.

\item Let $G$ be f.g. and residually finite. Then $\mathrm{ur}(G)$ is finite
if and only if $G\in\mathcal{S}$.
\end{itemize}

For most of these, see \cite{LR}, Chapter 5 and Chapter 1. The penultimate
claim is an easy consequence of \cite{LR}, \textbf{1.2.11}. The final claim is
\cite{MS1}, Theorem A.\bigskip

\textbf{Some known cases}

\medskip

\begin{proposition}
\label{P0}Let $G$ be a f.g. nilpotent-by-polycyclic group. If $G\in
\mathcal{U}$ then $G$ is a minimax group.
\end{proposition}

\begin{proof}
Let $N$ be a nilpotent normal subgroup of $G$ with $G/N$ polycyclic. It will
suffice to show that $G/N^{\prime}$ is minimax, so replacing $G$ by this
quotient we may assume that $N$ is abelian. Then $N$ is Noetherian as a
$G/N$-module, so the torsion subgroup $T$ of $N$ has finite exponent, $e$ say.
Let $\sigma$ be the set of prime divisors of $e$.

By P Hall's `generic freeness lemma' (cf. \cite{LR}, 7.1.6) $N/T$ has a free
abelian subgroup $F_{1}/T$ such that $N/F_{1}$ is a $\pi$-group for some
finite set of primes $\pi$. Then $F_{1}=T\times F$ where $F$ is free abelian,
and $N/F$ is a $\pi\cup\sigma$-group.

Let $p\notin\pi\cup\sigma$ be a prime. Then $N=FN^{p}$ and $N^{p}\cap F=F^{p}%
$, so $F/F^{p}\cong N/N^{p}$. Now $G/N^{p}$ is residually finite and the image
of $N/N^{p}$ in any finite quotient of $G/N^{p}$ has rank at most
$\mathrm{ur}_{p}(G)=r_{p}$; it follows that $\left\vert F/F^{p}\right\vert
=\left\vert N/N^{p}\right\vert \leq p^{r_{p}}$. Therefore $F$ has rank at most
$r_{p}.$ Hence for each prime $q\notin\pi\cup\sigma$ we have%
\[
\mathrm{ur}_{q}(G)\leq\mathrm{ur}(G/N)+\mathrm{ur}(F)\leq\mathrm{ur}%
(G/N)+r_{p}.
\]
It follows that $\mathrm{ur}(G)$ is finite, since $G/N$ is polycyclic and
$\pi\cup\sigma$ is finite.

As $G$ is residually finite it follows that $G$ is a minimax group. (For the
quoted properties of f.g. abelian-by-polycyclic groups, see for example
\cite{LR}, Chapters 4 and 7.)
\end{proof}

\bigskip

The upper $p$-rank of a group $G$ can equivalently be defined as \emph{the
rank of a Sylow pro-}$p$ \emph{subgroup }$P$\emph{ of} $\widehat{G}$, the
profinite completion of $G$, where for a profinite group $P$, the rank of $P$
is%
\[
\mathrm{r}(P)=\sup\{\mathrm{r}(P/N)\mid N\vartriangleleft P,~N\text{ open}\}.
\]

The pro-$p$ groups of finite rank are well understood (see \cite{DDMS}); in
particular, they are linear groups in characteristic $0$.

\begin{proposition}
\label{P1}\emph{(\cite{LS}, Window 8, Lemma 9)} Let $K$ be a f.g. residually
nilpotent group. Suppose that the pro-$p$ completion $\widehat{K}_{p}$ of $K$
has finite rank for some prime $p$. Then there exists a finite set of primes
$\pi$ such that the natural map%
\[
K\rightarrow\prod\nolimits_{q\in\pi}\widehat{K}_{q}%
\]
is injective.
\end{proposition}

This is the key to

\begin{theorem}
\label{rn}\emph{(\cite{S2}, Theorem 5) }Let $G$ be a f.g. group that is
virtually residually nilpotent. If $G\in\mathcal{U}$ then $G$ has finite upper rank.
\end{theorem}

\begin{proof}
It follows from Proposition \ref{P1} that $G$ has a subgroup $K$ of finite
index such that $K$ is residually (finite nilpotent of rank at most $r$); here
$r=\max_{q\in\pi}\mathrm{ur}_{q}(G)$. By a result mentioned above, we may also
take it that every finite quotient of $K$ is soluble. The main result of
\cite{S5} now shows that $K$ is virtually nilpotent-by-abelian (see also
\cite{LS}, Window 8, Corollary 5), and the result follows by Proposition
\ref{P0}.
\end{proof}

\bigskip

Let $\mathcal{H}$ denote the class of all groups $G$ with the property:
\emph{every virtually residually nilpotent quotient of} $G$ \emph{is a minimax
group}.\medskip

Theorem \ref{rn} shows that finitely generated groups in $\mathcal{U}$ belong
to $\mathcal{H}$. It is \emph{not} true that every f.g. soluble residually
finite group in $\mathcal{H}$ has finite upper rank:\bigskip

\begin{proposition}
\emph{(\cite{PS}, Proposition 10.1) }Let $p$ be a prime, let
\[
H=\left\langle x_{n}~(n\in\mathbb{Z});x_{n}^{p}=x_{n-1}\right\rangle
\]
be the additive group of $\mathbb{Z}[1/p]$ written multiplicatively, and let
$\tau$ be the automorphism of $H$ sending $x_{n}$ to $x_{n+1}$ for each $n$.
Extend $\tau$ to an automorphism of the group algebra $\mathbb{F}_{p}H$ and
then to an automorphism of $W=\mathbb{F}_{p}H\rtimes H=C_{p}\wr H$. Set
$G=W\rtimes\left\langle \tau\right\rangle $. Then

\begin{itemize}
\item $G$ is a $3$-generator residually finite abelian-by-minimax group

\item $G\in\mathcal{H}$

\item $\mathrm{ur}_{q}(G)=2$ for every prime $q\neq p$

\item $\mathrm{ur}_{p}(G)$ is infinite.
\end{itemize}
\end{proposition}

This shows, also, that the hypothesis of Conjecture A can't be weakened by
omitting finitely many primes.

Still, the strongest result so far obtained towards Conjecture A rests on a
consideration of certain groups in $\mathcal{H}$. It seems clear that the
trouble with the last example is due to the presence of `bad' torsion; if we
exclude this we obtain the following:

\begin{theorem}
\label{Hthm}\emph{(\cite{PS}, Theorem 3.2)} Let $G\in\mathcal{H}$ be f.g. and
residually finite. Suppose that $G$ has a metabelian normal subgroup $N$ with
$G/N$ polycyclic. Then $G/N^{\prime}$ is minimax. If $N^{\prime}$ has no $\pi
$-torsion where $\pi=\mathrm{spec}(G/N^{\prime})$ then $G$ is minimax.
\end{theorem}

(For a minimax group $H$, $\mathrm{spec}(H)$ denotes the (finite) set of
primes $p$ such $C_{p^{\infty}}$ is a section of $H$.)

\medskip

From this, it is relatively straightforward to deduce

\begin{theorem}
\label{metab}\emph{(cf. \cite{PS}, Corollary 3.3)} Let $G$ be a finitely
generated group that is nilpotent-by-abelian-by-polycyclic. If $G\in
\mathcal{U}$ then $G$ has finite upper rank.
\end{theorem}

\begin{proof}
We may assume that $G$ satisfies the hypotheses of Theorem \ref{Hthm}. Keeping
the notation there, put $A=N^{\prime}$, an abelian normal subgroup of $G$. For
a prime $p$ and $K\vartriangleleft_{f}G$ let $D_{p}(K)/(A\cap K)$ be the
$p^{\prime}$-component of the finite abelian group $A/(A\cap K)$. Then%
\[
\mathrm{r}_{p}(G/K)=\mathrm{r}_{p}(G/KD_{p}(K)).
\]
So if we set%
\[
D=\bigcap_{\substack{p\notin\pi\\K\vartriangleleft_{f}G}}KD_{p}(K),
\]
we have $\mathrm{ur}_{p}(G)=\mathrm{ur}_{p}(G/D)$ for all $p\notin\pi$.

Now $AD/D\cong A/(A\cap D)$ has no $\pi$-torsion, since each $A/(A\cap
KD_{p}(K))$ is a $p$-group. Clearly $G/D$ is residually finite, so Theorem
\ref{Hthm} applies to show that $G/D$ is a minimax group. Hence%
\[
\mathrm{ur}_{p}(G)=\mathrm{ur}_{p}(G/D)\leq\mathrm{ur}(G/D)<\infty
\]
for every $p\notin\pi$, and as $\pi$ is finite it follows that $\mathrm{ur}%
_{p}(G)$ is bounded over all primes $p$.
\end{proof}

\bigskip

The hypotheses in Theorem \ref{metab} seem rather restrictive. However, if we
could only replace `nilpotent-by-abelian' with `abelian-by-nilpotent' then we
could deduce the full force of Conjecture A; this is explained below.

\bigskip

\textbf{Modules of finite upper rank}

\medskip

Let $G$ be a counterexample to Conjecture A of least possible derived length,
$l$; we may assume that $G$ is residually finite. Let $A$ be maximal among
abelian normal subgroups of $G$ that contain $G^{(l-1)}$. Then $G/A$ is
residually finite (by an elementary lemma) and has finite upper rank, so $G/A$
is a minimax group. In particular, $G/A$ is virtually nilpotent-by-abelian and
so $G$ is abelian-by-nilpotent-by-polycyclic: this is the point of the final
remark in the preceding section.

Putting $\Gamma=G/A$ we consider $A$ as a $\Gamma$-module, written additively
as $A_{\Gamma}$.

If $B$ is a submodule of finite index in $A_{\Gamma}$ then $G/B$ is residually
finite (because $\mathcal{S}$ is extension-closed), whence%
\[
\mathrm{r}_{p}(A/B)\leq\mathrm{ur}_{p}(G)
\]
for each prime $p$; and it is clear that%
\[
\mathrm{ur}(G/B)\leq\mathrm{r}(A/B)+\mathrm{ur}(G/A).
\]

Let us define the upper rank of a $\Gamma$-module $M$ by $\mathrm{ur}%
(M)=\sup\{\mathrm{r}(M/B)\mid B\leq_{\Gamma}M,~M/B$ finite$\},$ and set%
\begin{align*}
\mathrm{ur}_{p}(M)  &  =\sup\{\mathrm{r}(M/B)\mid pM\leq B\leq_{\Gamma
}M,~M/B\text{ finite}\}\\
&  =\mathrm{ur}(M/pM).
\end{align*}
I will say that $M$ is a \emph{quasi-f.g.} $\Gamma$-module if there exists a
f.g. group $G$ that is an extension of $M$ by $\Gamma$. The preceding
observations now show that $A_{\Gamma}$ is a counterexample to

\bigskip

\textbf{Conjecture B} \emph{Let }$\Gamma$ \emph{be a f.g.\ residually finite
soluble minimax group and let }$M$ \emph{be a} \emph{quasi-f.g.} $\Gamma
$\emph{-module. If }$\mathrm{ur}_{p}(M)$ \emph{is finite for every prime} $p$
\emph{then} $M$ \emph{has finite upper rank.}

\medskip

Conversely, it is easy to see that if $M$ is a counterexample to Conjecture B
then the corresponding extension $G$ is a counterexample to Conjecture A. So
the two conjectures are equivalent.

Theorem \ref{metab} establishes Conjecture B for the special case where
$\Gamma$ is abelian-by-polycyclic. A reduction step in the proof is
Proposition 5.2 of \cite{PS}, which shows that $M$ contains a finitely
generated $\Gamma$-submodule $B$ such that the finite module quotients of $B$
are `nearly all' isomorphic to finite quotients of $M,$ and conversely. The
hypothesis that $\Gamma$ is abelian-by-polycyclic is used in the proof of this
reduction, but can be dispensed with; this is explained in the next section.
The main part of the proof, however, does depend on $\Gamma$ having an abelian
normal subgroup $A$ such that $\Gamma/A$ is polycyclic. Following a strategy
devised by P. Hall \cite{H} and further developed by Roseblade \cite{R}, one
examines the structure of $B$ as a module for the group ring $\mathbb{Z}A$,
with $\Gamma/A$ as a group of operators. The necessary module theory is
developed in \cite{S3} and \cite{S1}.

For the general case of Conjecture B, it would seem necessary to generalize
this machinery in one of two directions: either allow $A$ to be nilpotent
(rather than abelian), or allow $\Gamma/A$ to be minimax (rather than
polycyclic -- while still assuming $\Gamma/C_{\Gamma}(A)$ to be polycyclic, if
one takes $A$ inside the centre of the Fitting subgroup of $\Gamma$). Whether
either of these approaches is feasible remains unclear. Machinery relevant to
the first approach has been developed by Tushev \cite{T}. A major difficulty
with the second approach is the fact that the `generic freeness' property
mentioned above definitely fails when $\Gamma/A$ is not polycyclic, as
observed by Kropholler and Lorensen in \cite{KL}, Cor. 5.6. Other aspects of
the Hall-Roseblade theory have been usefully generalized by Brookes \cite{B}.

On the other hand, if one is seeking a counterexample to conjecture B, the
simplest candidate would seem to be the following group: Let $K$ be the
Heisenberg group over $\mathbb{Z}[1/2]$ and take $\Gamma=K\rtimes\left\langle
t\right\rangle $ where $t$ acts on a matrix by doubling the off-diagonal
entries (and multiplying the top right corner entry by $4$). Then $M$ could be
the quotient $\mathbb{Z}\Gamma/J$ where $J$ is a carefully constructed right
ideal: generators of $J$ should be chosen to ensure that $\mathbb{Z}\Gamma/J$
has finite upper $p$-rank for each prime $p$, but in such a way that these
ranks are unbounded.

\bigskip

\textbf{A possible reduction: quasi-f.g. modules.}

\bigskip

Let $\Gamma$ be a f.g.\ residually finite soluble minimax group. Then $\Gamma$
has a nilpotent normal subgroup $K$ such that $\Gamma/K$ is virtually abelian.
We fix a normal subgroup $Z$ of $\Gamma$ with $Z\leq\mathrm{Z}(K),$ and let
$R=\mathbb{Z}Z$ denote its group ring. For a multiplicatively closed subset
$\Lambda$ of $R$, an $R$-module $M$ is said to be $\Lambda$\emph{-torsion} if
every element of $M$ is annihilated by some element of $\Lambda$.

\begin{proposition}
\label{propo}Let $A$ be a quasi-f.g. $\Gamma$-module. Then $A$ has a finitely
generated $\Gamma$-submodule $B$ such that $A/B$ is $\Lambda$-torsion for each
$\Lambda$ of the form $R\smallsetminus L$ where $L$ is a maximal ideal of
finite index in $R$ not containing the augmentation ideal $(Z-1)R$.
\end{proposition}

Before giving the proof we note a corollary. For a $\Gamma$-module $M,$ let
$\mathcal{F}(M)$ denote the set of isomorphism types of finite quotient
$\Gamma$-modules of $M$.

\begin{corollary}
For $A$ and $B$ as above, we have%
\[
\mathcal{F}(A)\smallsetminus\mathcal{S}=\mathcal{F}(B)\smallsetminus
\mathcal{S}%
\]
where $\mathcal{S}$ consists of the finite $\Gamma$-modules that have a
composition factor on which $Z$ acts trivially.
\end{corollary}

This is essentially a formal consequence of the stated condition on $\Lambda
$-torsion, which implies that%
\[
AJ+B=A,~~AJ\cap B=BJ
\]
whenever $J$ is the annihilator in $R$ of some finite $\Gamma$-module not in
$\mathcal{S}$. Thus questions about the upper ranks of $A$ might be reduced to
questions about the upper ranks of the finitely generated module $B,$ if - by
some subsidiary argument - one could leave aside the quotients lying in
$\mathcal{S}$ (this is in principle the approach taken in \cite{PS}, \S \S 5, 6).

To establish the Proposition, we consider a f.g. group $E$ with an abelian
normal subgroup $A$ such that $E/A=\Gamma$. In $E$ there is a series of normal
subgroups%
\[
E\vartriangleright K_{1}\geq Z_{1}\geq A\geq\gamma_{c+1}(K_{1})[Z_{1},K_{1}]
\]
where $K_{1}/A=K$ is nilpotent of class $c$, say, and $Z_{1}/A=Z$. Now $Z$ is
an abelian minimax group, hence contains a finite subset $Y_{1}$ such that
$Z/\left\langle Y_{1}\right\rangle $ is divisible. Since $E/K_{1}\cong%
\Gamma/K$ is virtually abelian and $E$ is f.g., $K_{1}$ is finitely generated
as a normal subgroup of $E$; we choose a finite set $X=X^{-1}$ of normal
generators for $K_{1}$ and assume that $X$ contains a set $Y$ of
representatives for the elements of $Y_{1}$. Finally, let $S=S^{-1}$ be a
finite set of generators for $E$.

\begin{lemma}
\label{lambda}Let $L$ be a maximal ideal of finite index in $R=\mathbb{Z}Z$
not containing $Z-1.$ Then $\Lambda=R\smallsetminus L$ satisfies%
\begin{equation}
\left(  \Lambda^{g}+1\right)  \cap Y_{1}\neq\varnothing\label{inter}%
\end{equation}
for every $g\in E$.
\end{lemma}

\begin{proof}
Write $D=Z\cap(L+1)$. If (\ref{inter}) fails for $g$ then $D^{g}\supseteq
Y_{1}$ which implies $D^{g}=Z$ since $Z/\left\langle Y_{1}\right\rangle $ is
divisible while $\left\vert Z:D^{g}\right\vert $ is finite. Hence $D=Z$ and so
$L\supseteq Z-1$.
\end{proof}

\bigskip

Now we define $B$ to be the $E$-submodule of $A$ generated by the finite set%
\[
\left\{  \lbrack x,y],~[x^{s},y]\mid x\in X,~y\in Y,~s\in S\right\}  .
\]
We aim to show that if $\Lambda$ is a multiplicatively closed subset of $R$
satisfying (\ref{inter}) for every $g\in E$, then the $R$-module $A/B$ is
$\Lambda$-torsion; with Lemma \ref{lambda} this will complete the proof of
Proposition \ref{propo}.

Note that $\gamma_{i+1}(K_{1})$ is generated by the elements $v_{i}%
(\mathbf{x},\mathbf{w})^{g}$ for $g\in E$ and%
\[
v_{i}(\mathbf{x},\mathbf{w})=[x_{0},x_{1}^{w_{1}},\ldots,x_{i}^{w_{i}}],
\]
$x_{j}\in X,~w_{j}\in E$. Put%
\begin{align*}
A_{i}  &  =\left\langle [v_{i}(\mathbf{x},\mathbf{w}),z]^{g}\mid x_{j}\in
X,~z\in Y,~g,w_{j}\in E\right\rangle ,\\
B_{i}  &  =\left\langle [x^{v},y]^{g}\mid x\in X,~y\in Y,~g,v\in E,~l(v)\leq
i\right\rangle
\end{align*}
where $l(v)$ denotes the least $n$ such that $v=s_{1}\ldots s_{n}$ ($s_{j}\in
S$). Note that $B_{1}=B$.

\medskip\emph{Claim:} $A/A_{c}$ is $\Lambda$-torsion.

\medskip To see this, choose $y\in Y$ with $\overline{y}-1\in\Lambda$ where
$\overline{y}=Ay$. Then (mixing additive and multiplicative notation)%
\[
A(\overline{y}-1)^{c}=[A,_{c}y]\subseteq\gamma_{c+1}(K_{1})\subseteq A.
\]
Given a generator $v_{c}(\mathbf{x},\mathbf{w})^{g}$ of $\gamma_{c+1}(K_{1})$,
choose $z\in Y$ such that $\overline{z}^{g}-1\in\Lambda$. Then
\[
v_{c}(\mathbf{x},\mathbf{w})^{g}(\overline{z}^{g}-1)=[v_{c}(\mathbf{x}%
,\mathbf{w}),z]^{g}\in A_{c}.
\]
\medskip

\emph{Claim}: For $i>1$, $B_{i}/B_{i-1}$ is $\Lambda$-torsion.

\medskip To see this, say $b=[x^{\gamma u},y]^{g}$ is a generator of $B_{i}$
where $l(u)\leq i-1$. Choose $z\in Y$ such that $\overline{z}^{ug}-1\in
\Lambda$. Then%
\begin{align*}
(b(\overline{z}^{ug}-1))^{-g^{-1}y}  &  =[x^{\gamma u},y,z^{u}]^{-y^{-1}}\\
&  =[z^{u},x^{-\gamma u},y^{-1}]^{x^{\gamma u}}+[y^{-1},z^{-u},x^{\gamma
u}]^{z^{u}}.
\end{align*}
The first term lies in $B_{1}\leq B_{i-1}$ and the second term lies in
$B_{i-1}$. The claim follows since $B_{i-1}$ is $E$-invariant.

\medskip\emph{Claim}: Write $B_{\infty}=\bigcup_{j}B_{j}$. Then for $i>1$,
$A_{i}\subseteq B_{\infty}+A_{i-1}$.

\medskip To see this, let $x=(\mathbf{x}^{\prime},x)$ and $\mathbf{w}%
=(1,w_{1},\ldots)=(\mathbf{w}^{\prime},w)$ be $(i+1)$-tuples in $X,~E$
respectively, and let $z\in Y$. Then%
\begin{align*}
\lbrack v_{i}(\mathbf{x},\mathbf{w}),z]^{-x^{-w}}  &  =[v_{i-1}(\mathbf{x}%
^{\prime},\mathbf{w}^{\prime}),x^{w},z]^{-x^{-w}}\\
&  =[x^{w},z^{-1},v_{i-1}(\mathbf{x}^{\prime},\mathbf{w}^{\prime}%
)]^{z}+[z,v_{i-1}(\mathbf{x}^{\prime},\mathbf{w}^{\prime})^{-1},x^{w}%
]^{v_{i-1}(\mathbf{x}^{\prime},\mathbf{w}^{\prime})}.
\end{align*}
The first term lies in $B_{\infty}$ and the second term lies in $A_{i-1}$. The
claim follows since each of these modules is $E$-invariant.

The three claims together now imply that $A/B$ is $\Lambda$-torsion, and the
proof is complete.

\bigskip

\textbf{Further reductions}

\bigskip

Suppose that the pair $(\Gamma,M)$ furnishes a counterexample to Conjecture B,
where $M$ is finitely generated as a $\Gamma$-module. With quite a lot of
extra work, generalizing some ideas from \cite{S3}, one can establish

\begin{proposition}
The module $M$ has a torsion-free residually finite quotient $\widetilde{M}$
of infinite upper rank such that every proper, $\pi$-torsion-free residually
finite quotient of $\widetilde{M}$ has finite rank, where $\pi=\mathrm{spec}%
(\Gamma)$.
\end{proposition}

(Here $\mathrm{spec}(\Gamma)$ denotes the (finite) set of primes $p$ such that
$\Gamma$ has a section $C_{p^{\infty}}$.)

This reduces the problem to consideration of a `minimal counterexample', in a
rather weak sense. Whether this is any help is not clear, and there seems
little point in including the proof here.

Further results that may be relevant are obtained in \cite{KL1}; these can be
used to show that a module like our putative counterexample has many
finite-rank quotients that split as direct sums.

\end{document}